\documentclass[runningheads]{llncs}
\usepackage{graphicx}
\usepackage[symbol]{footmisc}

\usepackage{float}
\usepackage{mathtools}
\usepackage{amssymb, amsmath, latexsym}
\usepackage{nicefrac}
\usepackage{hhline}
\usepackage{makecell}
\usepackage{caption}
\usepackage{multirow}

\usepackage{colortbl}
\definecolor{bgcolor}{rgb}{0.8,1,1}
\definecolor{bgcolor2}{rgb}{0.8,1,0.8}
\usepackage{threeparttable}

\usepackage{pifont}
\usepackage[colorinlistoftodos,bordercolor=orange,backgroundcolor=orange!20,linecolor=orange,textsize=scriptsize]{todonotes}
\usepackage{algorithm}\usepackage{algpseudocode}

\allowdisplaybreaks
\graphicspath{ {images/} }
\def \R {\mathbb R}

\newtheorem{assumption}{Assumption}
  
\def\R{\mathbb{R}}

\def\R{\mathbb R}

\def\w{w}
\def\z{z}

\newcommand{\E}[1]{\mathbb{E}\left[#1\right]}
\newcommand{\Ek}[1]{\mathbb{E}_k\left[#1\right]}

\newcommand{\eqdef}{\vcentcolon=}

\newcommand{\norm}[1]{\left\| #1 \right\|}

\newcommand{\EndProof}{\begin{flushright}$\square$\end{flushright}}
\DeclareMathOperator{\dom}{\mathrm{dom}}

\newcommand{\F}{F}

\def\<#1,#2>{\langle #1,#2\rangle}
\newcommand{\sqn}[1]{\norm{#1}^2}

\begin{document} 
\title{Compression and Data Similarity: Combination of Two Techniques for Communication-Efficient Solving of Distributed Variational Inequalities\thanks{The work of A. Beznosikov was supported by the strategic academic leadership program 'Priority 2030' (Agreement  075-02-2021-1316 30.09.2021).
The work of A. Gasnikov was supported by the Ministry of Science and Higher Education of the Russian Federation (Goszadaniye), No. 075-00337-20-03, project No. 0714-2020-0005.}}
\titlerunning{Compression and Data Similarity for Distributed Variational Inequalities}
%
\author{Aleksandr Beznosikov\inst{1,2}
Alexander Gasnikov\inst{1,3,4}
}
\authorrunning{A. Beznosikov, A. Gasnikov}
%
\institute{Moscow Institute of Physics and Technology, Dolgoprudny, Russia \and
HSE University, Moscow, Russia \and
IITP RAS, Moscow, Russia \and
Caucasus Mathematical Center, Adyghe State University, Maikop, Russia
}
\maketitle              
\begin{abstract}

Variational inequalities are an important tool, which includes minimization, saddles, games, fixed-point problems. Modern large-scale and computationally expensive practical applications make distributed methods for solving these problems popular. Meanwhile, most distributed systems have a basic problem -- a communication bottleneck. There are various techniques to deal with it. In particular, in this paper we consider a combination of two popular approaches: compression and data similarity. We show that this synergy can be more effective than each of the approaches separately in solving distributed smooth strongly monotone variational inequalities. Experiments confirm the theoretical conclusions.


\end{abstract}
\section{Introduction}

Variational inequalities are a broad class of problems that have been widely studied for a long time. This is primarily due to the uniqueness of variational inequalities; they can describe various types of optimization problems, which, in turn, have many practical applications \cite{VIbook2003,Heinz}.  We can mention classic examples in economics  and game theory \cite{facchinei2007finite}, robust optimization \cite{BenTal2009:book}, non-smooth optimization \cite{nesterov2005smooth,nemirovski2004prox}, matrix factorization \cite{bach2008convex}, image denoising \cite{esser2010general,chambolle2011first}, supervised learning \cite{bach2011optimization}. In recent years, there has been a significant increase in research interest toward the study of variational inequalities due to new connections with GANs \cite{goodfellow2014generative}. In particular, the authors of \cite{daskalakis2017training,gidel2018variational,mertikopoulos2018optimistic,chavdarova2019reducing,pmlr-v89-liang19b,peng2020training} show that even if one considers the classical (in the variational inequalities literature) regime involving monotone and strongly monotone inequalities, it is possible to obtain insights, methods and recommendations useful for the GANs training.

Until recently, theoretical studies of methods for variational inequalities were carried out only in the non-distributed setting. The \texttt{Extra Gradient} / \texttt{Mirror Prox} method \cite{Korpelevich1976TheEM,Nemirovski2004,juditsky2008solving} became widely known and very popular. This alorithm for variational inequalities is key and basic (as Gradient Descent for minimization problems). But new practical problems have opened up new challenges. Indeed, the training of modern supervised machine learning models in general, and deep neural networks in particular, becomes more and more demanding. Solving such problems is almost impossible without a distributed approach with parallelization \cite{DMLsurvey}.

Meanwhile, the distributed approach has its bottlenecks. The main one is communication cost, as the transfer of information between computing devices takes considerably longer than local processes. This is why it is important not just to get a distributed version of e.g. \texttt{Extra Gradient} \cite{beznosikov2020local}, but a more effective method in terms of communication. The community already knows a number of approaches to 
communication efficient distributed optimization \cite{FEDLEARN,cocoa-2018-JMLR,Ghosh2020,MARINA}. For example, two such popular approaches are the compression of transmitted information, and the use of statistical similarity of local data on workers (if we spread the data uniformly among them).

\paragraph{}\hspace{-0.2cm}\textbf{Our contribution and related works.}
In this work we have combined two techniques for effective communications: compression and data similarity. Through this synthesis, we have obtained a method for distributed variational inequalities with better theoretical guarantees on the number of information transferred. See Table \ref{tab:comparison0}.

\renewcommand{\arraystretch}{2}
\begin{table}
    \centering
    \scriptsize
	\captionof{table}{Summary of complexities on the number of transmitted information for different approaches to communication bottleneck. \\
	{\em Notation:} $\mu$ = constant of strong monotonicity of the operator $F$, $L$ = Lipschitz constant of the operator $F$,  $\delta$ = similarity (relatedness) constant (Assumption \ref{as:delta}), $M$ = number of devices, $b$ = local data size, $\varepsilon$ = precision of the solution.
	}
    \label{tab:comparison0}
    \small
  \begin{threeparttable}
  \resizebox{\linewidth}{!}{
    \begin{tabular}{|c|c|c|c|c|}
    \hline
    \textbf{Method} & \textbf{Reference} & \textbf{Technique} & \textbf{Amount of information} & \textbf{If $\delta \sim \tfrac{L}{\sqrt{b}}$} \\
    \hline
    \texttt{Extra Gradient} & \cite{juditsky2008solving,beznosikov2020local} & & $O\left( \frac{L}{\mu}\log \frac{1}{\varepsilon} \right)$  &  $O\left( \frac{L}{\mu}\log \frac{1}{\varepsilon} \right)$ 
    \\ \hline
    \texttt{SMMDS} & \cite{beznosikov2021distributed} & similarity & $O\left( \frac{\delta}{\mu}\log \frac{1}{\varepsilon} \right)$  &  $O\left( \frac{1}{\sqrt{b}} \cdot \frac{L}{\mu}\log \frac{1}{\varepsilon} \right)$ 
    \\ \hline
   \texttt{MASHA} & \cite{beznosikov2021distributedcompr} & compression & $O\left( \frac{L}{\sqrt{M} \mu}\log \frac{1}{\varepsilon} \right)$  &  $O\left(\frac{1}{\sqrt{M}} \cdot \frac{L}{\mu}\log \frac{1}{\varepsilon} \right)$ 
    \\ \hline
    \texttt{Optimistic MASHA} & This work & 
    \makecell{compression \\ similarity}
     & $O\left( \left[\frac{L}{M \mu} + \frac{\delta}{ \sqrt{M}\mu}\right] \log \frac{1}{\varepsilon} \right)$ & $O\left( \left[\frac{1}{M} + \frac{1}{ \sqrt{Mb}}\right] \cdot \frac{L}{\mu} \log \frac{1}{\varepsilon} \right)$
    \\ \hline
    \end{tabular}
    }
    \end{threeparttable}
\end{table}

Separately from each other, similarity and compression techniques have already been investigated for both  particular minimization problems and general variational inequalities.

Different approaches with compression have been developed for minimization problems. Here we can  highlight the earliest and the simplest approach, in which compression operators were applied to \texttt{SGD}-type methods \cite{alistarh2017qsgd}. Further modifications with "memory" were presented in \cite{DIANA,horvath2019stochastic}. Then accelerated methods were introduced by authors of \cite{li2020acceleration}. The work \cite{MARINA}  was tried to look at compression through the variance reduction technique. There is now widespread research into practical modifications with biased operators and error compensation \cite{error_feedback,zheng2019communication,beznosikov2020biased,EF21}, bidirectional compression \cite{tang2019doublesqueeze,philippenko2020bidirectional}, partial participation \cite{horvath2020better,philippenko2020bidirectional,MARINA}, etc.
In the generality  of variational inequalities, compression methods were studied in  \cite{beznosikov2021distributedcompr}. One can note that our new method are ahead of \texttt{MASHA} from this paper (record results in terms of compression methods for variational inequalities at the moment). 

The literature on distributed minimization problems under similarity (relatedness) assumption is also vast. 
The paper \cite{arjevani2015communication} established lower communication complexity bounds. The authors of \cite{pmlr-v32-shamir14} proposed  the  mirror-descent based algorithm with data preconditioning. This technique was further accelerated by the inexact
damped Newton method \cite{pmlr-v37-zhangb15}, the Catalyst framework \cite{reddi2016aide} and the heavy ball momentum \cite{yuan2019convergence}. Higher order methods employing preconditioning were studied in   \cite{daneshmand2021newton,agafonov2021accelerated,tian2021acceleration}. Not for minimizations, but for variational inequalities and saddles, the similarity (relatedness) setup was considered by \cite{beznosikov2021distributed,kovalev2022optimal}.  In some cases, our estimates can also outperform the results from these works.



\section{Problem setup and assumptions}

\subsection{Variational inequality}

We consider variational inequalities (VI) of the  form
\begin{equation}\begin{aligned}
    \label{eq:VI}
    \text{Find} ~~ z^* &\in \R^d ~~ \text{such that} ~~ 
    \langle F(z^*), z - z^* \rangle + g(z) - g(z^*) \geq 0, ~~ \forall z \in \R^d,
\end{aligned}\end{equation}
where $F: \R^d \to \R^d $ is an operator, and $g: \R^d \to \R \cup \{ + \infty\}$ is a proper lower semicontinuous convex function. We assume that $F$ is distributed across $M$ workers/devices:
\begin{equation}
    \label{MK}
  F(z) \eqdef \frac{1}{M}\sum\limits_{m=1}^M F_m(z),
\end{equation}
where $F_m:\R^d \to \R^d$ for all $m \in \{1,2,\dots,M\}$.

\subsection{Examples}

To showcase the expressive power of the formalism \eqref{eq:VI}, we now give a  few examples of variational inequalities arising in machine learning.

\textbf{Example 1 [Convex minimization].} Consider the composite minimization problem:
\begin{align}
\label{eq:min}
\min_{z \in \R^d} f(z) + g(z),
\end{align}
where $f$ is typically a main term, and $g$ is a regularizer or an indicator function (e.g., if we want to consider the problem on some set). 
If we put $F(z) \eqdef \nabla f(z)$, then it can be proved that $z^* \in \dom g$ is a solution for \eqref{eq:VI} if and only if $z^* \in \dom g$ is a solution for \eqref{eq:min}. 


\textbf{Example 2 [Convex-concave saddle point problems].} Consider the convex-concave saddle point problem
\begin{align}
\label{eq:minmax}
\min_{x \in \R^{d_x}} \min_{y \in \R^{d_y}} f(x,y) + g_1 (x) + g_2(y),
\end{align}
where $g_1$ and $g_2$ can also be interpreted as regularizers or indicators. 
If we put $F(z) \eqdef F(x,y) = [\nabla_x f(x,y), -\nabla_y f(x,y)]$ and $g(z) = g(x,y) = g_1 (x) + g_2(y)$, then it can be proved that $z^* \in \dom g$ is a solution for \eqref{eq:VI} if and only if $z^* \in \dom g$ is a solution for \eqref{eq:minmax}. 

While minimization problems are widely investigated separately from variational inequalities, saddles are very often studied together with variational inequalities. In particular, lower bounds for the former are also valid for the latter. Moreover, upper bounds for variational inequalities are valid for saddle point problems. However, perhaps more importantly, these lower and upper bounds coincide. This is in contrast to minimization, where the lower bounds are weaker. 

\subsection{Assumptions}

\begin{assumption}[Lipschitzness] \label{as:Lipsh}
The operator $F$ is $L$-Lipschitz continuous, i.e. for all $u, v \in \R^d$ we have
\begin{equation}
\label{eq:Lipsh}
\| F(u)-F(v) \|  \leq L\|u - v\|.
\end{equation}
\end{assumption}
For problems \eqref{eq:min} and \eqref{eq:minmax},  $L$-Lipschitzness of the operator means that the functions $f(z)$ and $f(x,y)$ are $L$-smooth.

\begin{assumption}[Strong monotonicity]\label{as:strmon}
The operator $F$ is $\mu$-strongly monotone, i.e. for all $u, v \in \R^d$ we have
\begin{equation}
\label{eq:strmon}
\langle F(u) - F(v); u - v \rangle \geq \mu \| u-v\|^2.
\end{equation}
\end{assumption}
For problems \eqref{eq:min} and \eqref{eq:minmax}, strong monotonicity of $F$ means strong convexity of $f(z)$ and strong convexity-strong concavity of $f(x,y)$.

\begin{assumption}[$\delta$-relatedness]\label{as:delta}
Each operator $F_m$ is $\delta$-related. It means that each operator $F_m - F$ is $\delta$-Lipschitz continuous, i.e. for all $u, v \in \R^d$ we have
\begin{equation}
\label{eq:delta}
\| F_m(u) - F(u)- F_m(v) + F(v)\|  \leq \delta\|u - v\|.
\end{equation}
\end{assumption}
While Assumptions \ref{as:Lipsh} and \ref{as:strmon} are basic and widely known, Assumption \ref{as:delta} requires further comments. This assumption goes back to the conditions of data similarity.
In more detail, we consider distributed minimization \eqref{eq:min} and saddle point \eqref{eq:minmax} problems:
\begin{equation*}
f(z) = \frac{1}{M} \sum\limits_{m=1}^M f_m(z), \quad f(x,y) = \frac{1}{M} \sum\limits_{m=1}^M f_m(x,y),
\end{equation*}
and assume that for minimization local and global hessians are \textit{$\delta$-similar} \cite{pmlr-v32-shamir14,pmlr-v37-zhangb15,yuan2019convergence,hendrikx2020statistically,kovalev2022optimal}:
\begin{equation*}
\| \nabla^2 f(z) - \nabla^2 f_m(z)\| \leq \delta,
\end{equation*}
and for saddles second derivatives are differ by $\delta$ \cite{beznosikov2021distributed,kovalev2022optimal}:
\begin{align*}
\| \nabla^2_{xx} f(x,y) - \nabla^2_{xx} f_m(x,y)\| \leq \delta, \\
\| \nabla^2_{xy} f(x,y) - \nabla^2_{xy} f_m(x,y)\| \leq \delta, \\
\| \nabla^2_{yy} f(x,y) - \nabla^2_{yy} f_m(x,y)\| \leq \delta.
\end{align*}
It turns out that if we look at machine learning and the data is u distributed between devices, it can be proven \cite{tropp2015introduction,hendrikx2020statistically} that $\delta = \mathcal{\tilde O}\left( \tfrac{L}{\sqrt{b}}\right)$, where $b$ is the number of local data points on each of the workers.

\section{Main part}

Our new algorithm \texttt{Optimistic MASHA}, as well as \texttt{MASHA} from \cite{beznosikov2021distributedcompr}(the only compressed algorithm for variational inequalities already presented in the community), is based on the ideas of negative momentum and variance reduction technique \cite{alacaoglu2021stochastic,kovalev2022optimalvi}.

\begin{algorithm}[h]
	\caption{\texttt{Optimistic MASHA}}
	\label{alg:optmasha}
	\begin{algorithmic}[1]
\State
\noindent {\bf Parameters:}  Stepsize $\gamma>0$, parameter $\tau$, number of iterations $K$.\\
\noindent {\bf Initialization:} Choose  $z^0 = w^0 \in \mathcal{Z}$. \\
Server  sends to devices $z^0 = w^0 = w^{-1}$ and devices compute $F_m(z^0)$ and send to server and get $F(z^0)$
\For {$k=0,1, 2, \ldots, K-1$ }
\For {each device $m$ in parallel}
\State  Compute $F_m(z^{k})$
\State $\delta^k_m = F_m(z^k) - F_m (w^{k-1}) + \alpha[F_m(z^k) - F_m(z^{k-1})]$
\State Send  $Q_m\left( \delta^k_m\right)$ to server
\EndFor
\For {server}
\State Compute $ \frac{1}{M} \sum \limits_{m=1}^M Q_m(\delta^k_m)$ and send to devices
\State Sends to devices $b_k$: 1 with probability $\gamma$, 0 with. probability $1 - \gamma$ \label{alg1:1bit}
\EndFor
\For {each device $m$ in parallel}
\State $\Delta^k =  \frac{1}{M} \sum \limits_{m=1}^M Q^{\text{dev}}_m(\delta^k_m) + F(w^{k-1})$
\State  $z^{k+1} = \text{prox}_{\eta g} \left(z^k + \gamma (w^k - z^k)- \eta \Delta^k\right)$ \label{alg1_q_zk1}
\If {$b_k = 1$}
\State $w^{k+1} = z^{k}$ \label{alg1:wk_zk}
\State Compute $F_m(w^{k+1})$ and send it to server
\State Get $F(w^{k+1})$ as a response from server
\Else 
\State $w^{k+1} = w^k$  \label{alg1:wk_wk}
\EndIf
\EndFor
\EndFor
	\end{algorithmic}
\end{algorithm}

At the beginning of each iteration of \texttt{Optimistic MASHA}, each device sends the compressed version of $\delta_m^k$ to the server, and the server does a reverse broadcast.  It is also possible to compress the messages coming from the server to the devices, but in practical cases there is very often no need for compression in this case since the transfer process from the server takes less time than the sending from the devices to the server \cite{alistarh2017qsgd,DIANA,beznosikov2020biased}. Also, the workers receive a bit of information $b_k$. This bit is generated randomly on the server and is equal to 1 with probability $\gamma$ (where $\gamma$ is small). Note that $b_k$ can be generated locally, it is enough to use the same random generator and set the same seed on all devices. Then, all devices make a final update on $z^{k+1}$ using $\Delta^k$. 
One can notice that to compute $\Delta^k$ we need to know $F(w^{k-1})$ (the full operator over all nodes). Then, at first glance, it seems that we need to always send the uncompressed operators. But that is not the case. It is enough to look at the $w^{k+1}$ update. We put $w^{k+1} = z^{k+1}$ if $b_k = 1$ or save it from the previous iteration $w^{k+1} = w^{k}$ if $b_k = 0$. In the case where $w^{k+1} = z^{k+1}$, we need to exchange the full values of $F_m (w^{k+1})$ to make the value $F(w^{k+1})$ known to all nodes at the next iteration, but we do this rarely, with small probability $\gamma$.

Unlike \texttt{MASHA}, we don not use arbitrary $Q_m$ compressors on the devices, but a specific set $\{Q_m\}$, the so-called Permutation compressors, introduced in \cite{szlendak2021permutation}. 

\begin{definition}[Permutation compressors \cite{szlendak2021permutation}]
  \label{def:PermK}
\\
$\bullet$ \textbf{for $d \geq M$.}
Assume that $d \geq M$ and $d  = q M$, where $q\geq 1$ is an integer. Let $\pi = (\pi_1,\dots,\pi_d)$ be a random permutation of $\{1, \dots, d\}$. Then for all $u \in \R^d$ and each $m\in \{1,2,\dots,M\}$ we define
 \begin{equation*}
 Q_m(u) \eqdef M \cdot \sum \limits_{i = q (m - 1) + 1}^{q m} u_{\pi_i} e_{\pi_i}.\end{equation*}
\\
$\bullet$ \textbf{for $d \leq M$.} Assume that $M \geq d,$ $M > 1$ and $M  = q d$, where $q\geq 1$ is an integer. 
Define the multiset $S \eqdef \{1, \dots, 1, 2, \dots, 2, \dots, d, \dots, d\}$, where each number occurs precisely $q$ times. Let $\pi=(\pi_1,\dots,\pi_{M})$ be a random permutation of $S$. Then for all $u \in \R^d$ and each $m\in \{1,2,\dots,M\}$ we define  
 \begin{equation*}
    Q_m(u) \eqdef d u_{\pi_m}e_{\pi_m}.
\end{equation*}
\end{definition}

The essence of such compressors is that their behavior is related to each other. For example, in the case when $d \geq M$ and $d = Mq$, each device transmits only $q$ components of the full gradient and, importantly, these components are unique to that device. To make such a connection between compressors, one can set the same random seeds on the devices to generate permutations. The use of the Permutation compressors allows us to simultaneously benefit from both the data similarity and the compression of the transmitted information.
Briefly, the idea can be described as follows. Since we have $\delta$-related ($\delta$-similar) operators $\{F_m\}$, in a rough approximation we can assume that $F_m \approx F$ and then $\delta^k_m \approx \delta^k = F(z^k) - F (w^{k-1}) + \alpha[F(z^k) - F(z^{k-1})]$. 
But when we compress $\{\delta^k_m\}$ with the Permutation compressors, we end up with $\frac{1}{M} \sum_{m=1}^M Q^{\text{dev}}_m(\delta^k_m)$ that is close to uncompressed $\delta^k$.
In the meantime, we transmit $M$ times less information.

The following statements give a formal convergence of Algorithm \ref{alg:optmasha}. To begin, we give the lemma about the compressors from \cite{szlendak2021permutation}.

\begin{lemma} [see \cite{szlendak2021permutation}] \label{lem:perm}
    The Permutation compressors 
from Definition~\ref{def:PermK} are unbiased and satisfy
\begin{equation}
\label{eq:AB}\E{ \norm{ \frac{1}{M} \sum \limits_{m=1}^M Q_m(a_m) - \frac{1}{M}\sum\limits_{m=1}^M a_m }^2 } \leq \frac{1}{M}\sum\limits_{m=1}^M \norm{ a_m -  \frac{1}{M}\sum\limits_{i=1}^M a_i }^2\end{equation}
for all $a_1,\dots,a_M\in \R^d$.
\end{lemma}

Next, we present the main theorem.

\begin{theorem}\label{th:ALg1_conv}
	Consider the problem \eqref{eq:VI} under Assumptions~\ref{as:Lipsh}, \ref{as:strmon} and \ref{as:delta}. Let  $\{\z^k\}$ be the sequence generated by Algorithm~\ref{alg:optmasha} with the compressors from Definition \ref{def:PermK} and parameters 
	\begin{align*}
		0 < \gamma  \leq \frac{1}{8}, \quad \alpha = \frac{1}{2}, \quad \eta = \min\left\{\frac{\sqrt{\alpha\gamma}}{2 \delta}, \frac{1}{8(L + \delta)}\right\}.
	\end{align*}
	Then, given $\varepsilon>0$, the number of iterations for 
$\|\z^k - \z^*\|^2 \leq \varepsilon$ is 
\begin{equation*}
    O\left( \left[\frac{1}{\gamma} +\frac{L}{\mu} + \frac{\delta}{ \sqrt{\gamma}\mu}\right] \log \frac{1}{\varepsilon} \right).
\end{equation*}
\end{theorem}

The proof of this Theoerm is given in Appendix \ref{sec:proof_app}.

The resulting convergence estimate depends on the parameter $\gamma$. Let us find the way to choose it. In average, once per $\tfrac{1}{\gamma}$ iterations (when $b_k = 1$), we send uncompressed information. Hence, we can find the best option for $\gamma$. At each iteration the device sends $\mathcal{O}\left( \tfrac{1}{M} + \gamma \right)$ bits  -- each time information compressed by $\tfrac{1}{M}$ times and with probability $\gamma$ we send the full package. We get the optimal choice of $\gamma$:
\begin{corollary} \label{main_cor1_q}
Under the conditions of Theorem \ref{th:ALg1_conv}, and with $\gamma = \tfrac{1}{M}$, \texttt{Optimistic MASHA} with the Permutation compressors has the following estimate on the total number of transmitted information to find $\varepsilon$-solution
\begin{equation*}
    O\left( \left[\frac{L}{M \mu} + \frac{\delta}{ \sqrt{M}\mu}\right] \log \frac{1}{\varepsilon} \right).
\end{equation*}
\end{corollary}

\textbf{Discussion of the results in terms of compression.} As noted earlier, under conditions of uniformly distributed data, the parameter $\delta = \mathcal{\tilde O}\left( \tfrac{L}{\sqrt{b}}\right)$, where $b$ is the number of local data points on each of the devices. Note that a typical situation is when $b \geq M$. Then, the estimate from Corollary \ref{main_cor1_q} can be rewritten as 
\begin{equation*}
    O\left( \frac{L}{M \mu} \log \frac{1}{\varepsilon} \right).
\end{equation*}
State of the art methods for solving variational inequalities \cite{juditsky2008solving,beznosikov2020local}, which are also optimal algorithms in terms of the number of communications (but not the number of transmitted information) give the next estimate on the number of transmitted information
\begin{equation*}
    O\left( \frac{L}{\mu} \log \frac{1}{\varepsilon} \right).
\end{equation*}
\texttt{MASHA} can guarantee the following bound for the amount of transferred information:
\begin{equation*}
    O\left( \frac{L}{\sqrt{M} \mu} \log \frac{1}{\varepsilon} \right).
\end{equation*}
This shows that our result is $M$ times better than the uncompressed methods, and better than \texttt{MASHA} (which does not use $\delta$-relatedness) by $\sqrt{M}$ times. 

\textbf{Discussion of the results in terms of data similarity.} The algorithm from \cite{beznosikov2021distributed} using similarity for variational inequalities (in fact, for saddle point problems) has the following estimate for the number of information to be forwarded
\begin{equation*}
    O\left( \frac{\delta}{\mu} \log \frac{1}{\varepsilon} \right).
\end{equation*}
If $M \geq b$, the estimate from Corollary \ref{main_cor1_q} is transformed as follows
\begin{equation*}
    O\left( \left[\frac{L}{\sqrt{M} \sqrt{b} \mu} + \frac{\delta}{ \sqrt{M}\mu}\right] \log \frac{1}{\varepsilon} \right) = O\left( \frac{\delta}{ \sqrt{M}\mu} \log \frac{1}{\varepsilon} \right).
\end{equation*}
This result is $\sqrt{M}$ times better than from \cite{beznosikov2021distributed}.

\section{Experiments}

The aim of our experiments is to test the results of Corollary \ref{main_cor1_q}, namely the dependence of convergence on the parameter $\delta$. To be able to vary the parameter $\delta$ we conduct our experiments on a distributed bilinear problem, i.e., the problem \eqref{eq:minmax} with 
\begin{align}
    \label{bilinear}
     f_m(x,y) \eqdef  x^\top A_m y + a^\top_m x + b^\top_m y + \frac{\lambda}{2} \| x\|^2 -  \frac{\lambda}{2} \|y\|^2,
\end{align}
where $A_m \in \R^{d \times d}$, $a_m, b_m \in \R^d$. This problem is $\lambda$-strongly convex--strongly concave and, moreover, $L$-smooth with $ L = \|A \|_2$ for $A = \tfrac{1}{M} \sum_{m=1}^M A_m$.  We take $M = 10$, $d=100$ and generate matrix $A$ (with $\|A \|_2 \approx 100$) and vectors $a_m, b_m$ randomly. We also generate matrices $B_m$ such that all elements of these matrices are independent and have an unbiased normal distribution with variance $\sigma^2$. Using these matrices, we compute $A_m = A + B_m$. It can be considered that $\delta \sim \sigma$. In particular, we run three experiment setups: with small $\sigma \approx \tfrac{\|A \|_2}{100}$, medium $\sigma \approx \tfrac{\|A \|_2}{10}$ and big $\sigma \approx \|A \|_2$. $\lambda$ is chosen as $\tfrac{\|A\|_2}{10^5}$.

We use the new algorithm -- \texttt{Optimistic MASHA}, the existing compression algorithm \texttt{MASHA} \cite{beznosikov2021distributedcompr}, and the classic uncompressed \texttt{Extra Gradient} \cite{juditsky2008solving,beznosikov2020local} as competitors. 
In \texttt{Optimistic MASHA} and \texttt{MASHA} we use the Permutation compressors. All methods are tuned as outlined in the theory of the corresponding papers.

\begin{figure}[h]
\begin{minipage}{0.33\textwidth}
  \centering
\includegraphics[width =  \textwidth ]{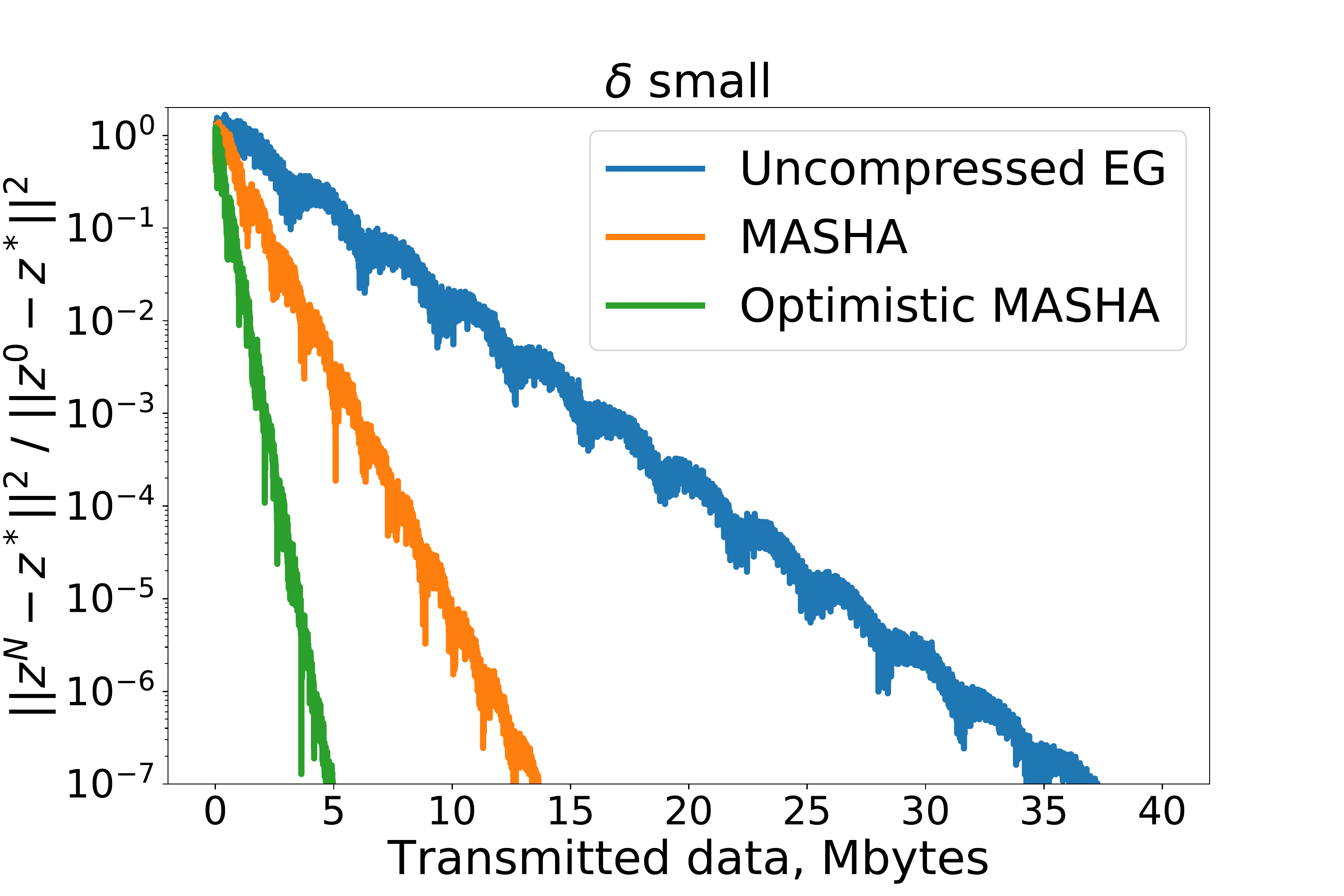}
\end{minipage}%
\begin{minipage}{0.33\textwidth}
  \centering
\includegraphics[width =  \textwidth ]{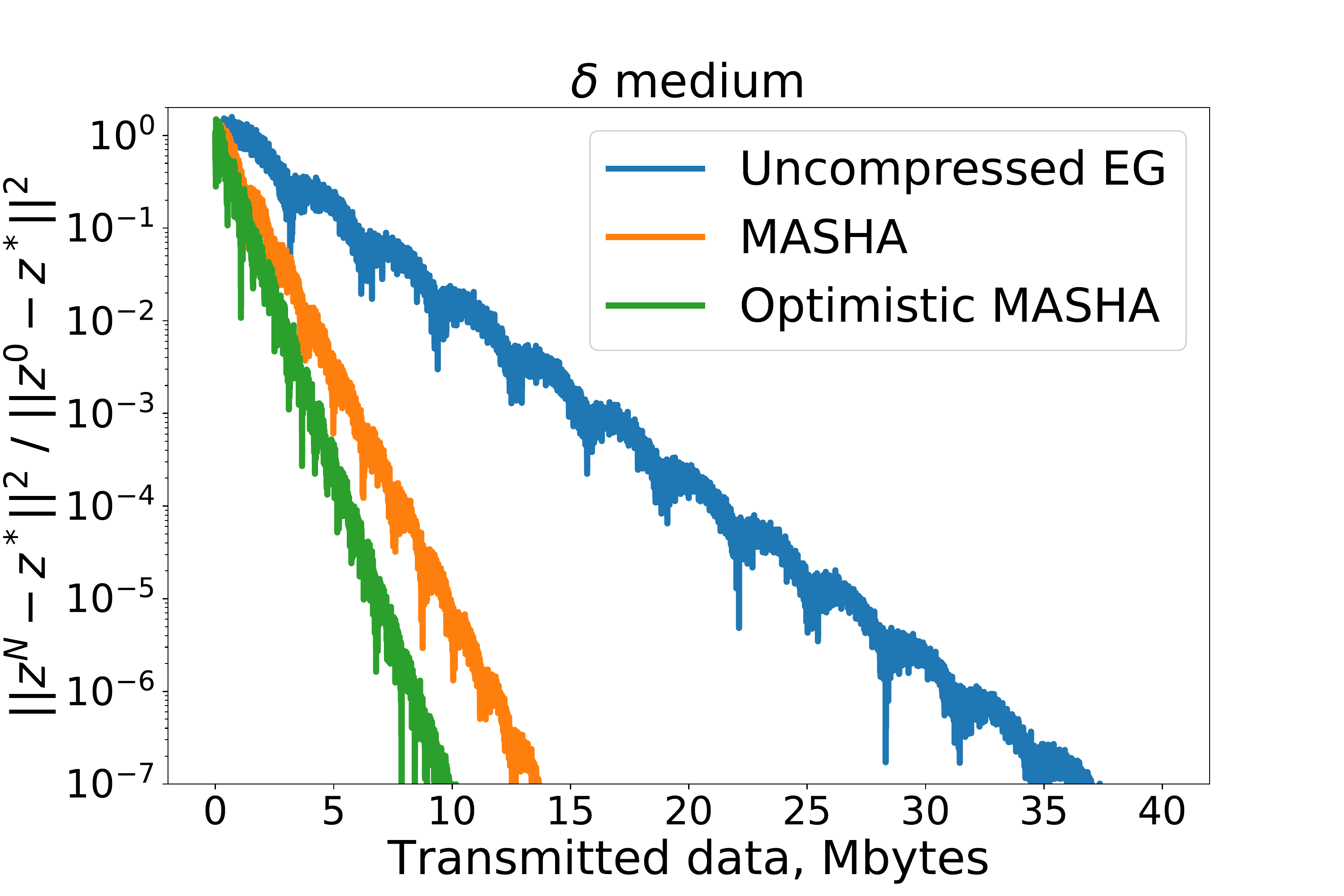}
\end{minipage}%
\begin{minipage}{0.33\textwidth}
  \centering
\includegraphics[width =  \textwidth ]{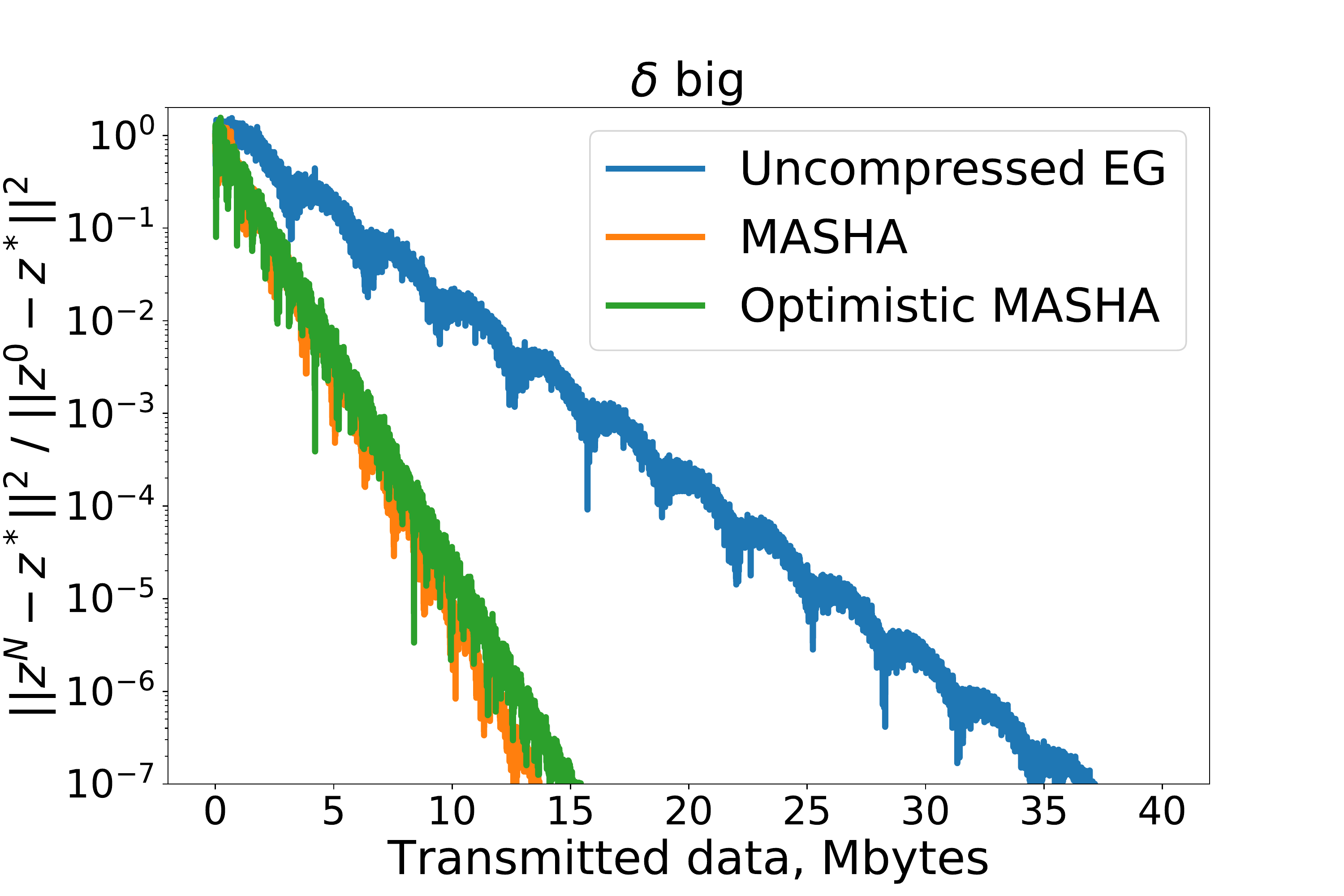}
\end{minipage}%
\\
\begin{minipage}{0.33\textwidth}
  \centering
(a) small $\delta$
\end{minipage}%
\begin{minipage}{0.33\textwidth}
 (b) medium $\delta$
\end{minipage}%
\begin{minipage}{0.33\textwidth}
  (c) big $\delta$
\end{minipage}%
\caption{\small Bilinear problem  \eqref{bilinear}: Comparison of state-of-the-art methods with compression for variational inequalities for small, medium and big similarity parameters.}
    \label{fig:min}
\end{figure}
See Figure \ref{fig:min} for the results. For small $\delta$ \texttt{Optimistic MASHA} is about $\sqrt{M}$ times superior to \texttt{MASHA}, and also outperforms the uncompressed method by a factor of $M$. With increasing $\delta$ \texttt{Optimistic MASHA} comes close to \texttt{MASHA} in its convergence.

\section{Conclusion}

In this paper, we considered distributed methods for solving the variational inequality problem. We presented the new method \texttt{Optimistic MASHA}. By combining two techniques: compression and data similarity, our method allows us to significantly reduce the number of information transmitted during communications. Experiments confirm the theoretical conclusions.

\bibliographystyle{splncs04}
\bibliography{ltr}

\newpage 

\appendix



\section{Proof of Theorem \ref{th:ALg1_conv}} \label{sec:proof_app}

\begin{lemma} \label{var_lem_fix}
    Consider the problem \eqref{eq:VI} under Assumption~\ref{as:delta}. Let  $\{z^k\}$ be the sequence generated by Algorithm~\ref{alg:optmasha} with compressors from Definition \ref{def:PermK}. Then, the following inequality holds:
	\begin{equation}\label{vrvi:eq:1}
		\E{\sqn{\Delta^k - \Ek{\Delta^k}}} \leq 2\delta^2\E{\sqn{z^{k} - w^{k-1}} + \alpha^2\sqn{z^k- z^{k-1}}},
	\end{equation}
	where $\Ek{\Delta^k}$ is equal to
	\begin{equation}\label{vrvi:eq:3}
		\Ek{\Delta^k} = F(z^k) + \alpha(F(z^k) - F(z^{k-1})).
	\end{equation}
\end{lemma}
\begin{proof}
    Due to the unbiasedness of $Q_m$, we can make sure that \eqref{vrvi:eq:3} is correct. Then using definitions of $\delta^k$ and $\Delta^k$ from Algorithm~\ref{alg:optmasha}, we get
	\begin{align*}
		&\Ek{\sqn{\Delta^k - \Ek{\Delta^k}}} 
		\\&=
		\Ek{\sqn{\frac{1}{M} \sum \limits_{m=1}^M Q_m(\delta^k_m) + F(w^{k-1}) -F(z^k) - \alpha(F(z^k) - F(z^{k-1}))}} 
		\\&=
		\Ek{\sqn{\frac{1}{M} \sum \limits_{m=1}^M Q_m(\delta^k_m)- \frac{1}{M} \sum \limits_{m=1}^M \delta^k_m}}.
	\end{align*}
By definition of $Q^{\text{dev}}_m$ (Lemma \ref{lem:perm}), we obtain	
	\begin{align*}
		&\Ek{\sqn{\Delta^k - \Ek{\Delta^k}}} 
		\\&\leq 
		\frac{1}{M} \sum \limits_{m=1}^M \sqn{ \delta^k_m - \frac{1}{M} \sum \limits_{m=1}^M \delta^k_m}
		\\&=
		\frac{1}{M} \sum \limits_{m=1}^M \bigg\|F_m(z^k) - F_m (w^{k-1}) + \alpha[F_m(z^k) - F_m(z^{k-1})] + F(w^{k-1})
		\\&\hspace{0.4cm}
		- F(z^k) - \alpha(F(z^k) - F(z^{k-1}))\bigg\|^2
		\\&=
		\frac{1}{M} \sum \limits_{m=1}^M \bigg\|F_m(z^k) - F(z^k) - F_m (w^{k-1}) + F(w^{k-1})  
		\\&\hspace{0.4cm}
		+ \alpha[F_m(z^k) - F(z^k) - F_m(z^{k-1}) + F(z^{k-1})] \bigg\|^2
		\\&\leq 
		\frac{2}{M} \sum \limits_{m=1}^M \sqn{F_m(z^k) - F(z^k) - F_m (w^{k-1}) + F(w^{k-1}) }
		\\&\hspace{0.4cm}
		+\frac{2\alpha^2}{M} \sum \limits_{m=1}^M \sqn{F_m(z^k) - F(z^k) - F_m(z^{k-1}) + F(z^{k-1})}.
	\end{align*}
Assumption on $\delta$-relatedness gives
	\begin{align*}
		\Ek{\sqn{\Delta^k - \Ek{\Delta^k}}} \leq&
		2 \delta^2 \sqn{z^k - w^{k-1} } 
		+2\alpha^2 \delta^2 \sqn{z^k - z^{k-1}}.
	\end{align*}
This concludes the proof of \eqref{vrvi:eq:1}.
\EndProof
\end{proof}
Before proving the main lemma of this section, we define the following Lyapunov function:
	\begin{align}
	\label{eq:Lf_fixed}
		\Psi^{k+1}&\eqdef (1+ 2 \mu \eta)\sqn{\z^{k+1} - \z^*} + \frac{\gamma + \eta\mu}{\gamma}\sqn{\w^{k+1} - \z^*}
	\\&\quad
	+ 2\eta\<\F(\z^k) - \F(\z^{k+1}), \z^{k+1} - \z^*>
	+ \gamma\sqn{\w^k - \z^{k+1}} +\frac{1}{8} \sqn{\z^{k+1}- \z^{k}} \notag.
	\end{align}

\begin{lemma}\label{lem:fixed_key}
	Consider the problem \eqref{eq:VI} under Assumptions~\ref{as:Lipsh}, \ref{as:strmon} and \ref{as:delta}. Let  $\{z^k\}$ be the sequence generated by Algorithm~\ref{alg:optmasha} with compressors from Definition \ref{def:PermK} and parameters 
	\begin{align*}
		0<\gamma \leq \frac{1}{8}, \quad \alpha \in (0;1), \quad \eta \leq \min\left\{\frac{\sqrt{\alpha\gamma}}{2 \delta}, \frac{1}{8(L + \delta)}\right\}.
	\end{align*}
	Then, after $k$ iterations we get
\begin{align*}
	\E{\frac{1}{2} \|z^k - z^* \|^2} \leq \max\left[ \left(1- \frac{\mu \eta}{2}\right); \left(1- \frac{1}{\frac{1}{\eta\mu} + \frac{1}{\gamma}}\right) ; \alpha ; \frac{1}{2} \right] \cdot^k \cdot \Psi^0.
\end{align*}
\end{lemma}
\begin{proof} 
We start from line \ref{alg1_q_zk1} of Algorithm~\ref{alg:optmasha}
\begin{align*}
		\sqn{z^{k+1} - z^*}
		=&
		\sqn{z^k - z^*} + 2\<z^{k+1} - z^k,z^{k+1} - z^*> - \sqn{z^{k+1} - z^k}
		\\=&
		\sqn{z^k - z^*} + 2\gamma \<w^k - z^k, z^{k+1} - z^*> 
		\\&- 2\eta \<\Delta^k - F(z^*),z^{k+1} - z^*>  - \sqn{z^{k+1} - z^k} 
		\\& -
		2\<z^k + \gamma (w^k - z^k) - \Delta^k - z^{k+1} + \eta F(z^*),z^{k+1} - z^*>.
\end{align*}
Optimality condition for \eqref{eq:VI} it follows, that
	\begin{equation*}
		-F(z^*) \in \partial g(z^*).
	\end{equation*}
From update (line \ref{alg1_q_zk1}) for $z^{k+1}$ of Algorithm~\ref{alg:optmasha} it follows, that
	\begin{equation*}
		z^k + \gamma(w^k - z^k) - \eta \Delta^k - z^{k+1} \in \partial (\eta g)(z^{k+1}).
	\end{equation*}
Hence, from  monotonicity of $\partial g(\cdot)$ we get
	\begin{align*}
	\E{\sqn{z^{k+1} - z^*}}
	&\leq
	\E{\sqn{z^k - z^*}} + 2\gamma \E{\<w^k - z^k, z^{k+1} - z^*>}
	\\&\quad -2\eta\E{\<\Delta^k - F(z^*), z^{k+1} - z^*>}  - \E{\sqn{z^{k+1} - z^k}}
	\\&= 
	\E{\sqn{z^k - z^*}} + 2\gamma \E{\<w^k - z^*, z^{k+1} - z^*>}
	\\&\quad - 2\gamma\E{\<z^k - z^*, z^{k+1} - z^*>}
	\\&\quad
	-2\eta\E{\<\Delta^k - F(z^*), z^{k+1} - z^*>} 
	 - \E{\sqn{z^{k+1} - z^k}}
	\\&=
	\E{\sqn{z^k - z^*}}
	\\&\quad
	+ \gamma \E{\sqn{w^k - z^*} + \sqn{z^{k+1} - z^*} - \sqn{z^{k+1} - w^k}} 
	\\&\quad- \gamma\E{\sqn{z^{k+1} - z^*} + \sqn{z^k - z^*} - \sqn{z^{k+1} - z^k}}
	\\&\quad-2\eta\E{\<\Delta^k - F(z^*), z^{k+1} - z^*>}  - \E{\sqn{z^{k+1} - z^k}}
	\\&=
	\E{\sqn{z^k - z^*}} + \gamma\E{\sqn{w^k - z^*}} - \gamma\E{\sqn{z^k - z^*}} 
	\\&\quad
	- \gamma\E{\sqn{w^k - z^{k+1}}}
	-
	2\eta\E{\<\Delta^k - F(z^*), z^{k+1} - z^*>}
	\\&\quad
	 - (1-\gamma)\E{\sqn{z^{k+1} - z^k}}.
	\end{align*}
In previous we also use the simple fact $\| a + b\|^2 = \|a\|^2 + 2\langle a; b \rangle + \|b\|^2$ twice. Small rearrangement gives
	\begin{align*}
	\E{\sqn{\z^{k+1} - \z^*}}
	&\leq
	\E{\sqn{\z^k - \z^*}} + \gamma\E{\sqn{\w^k - \z^*}} - \gamma\E{\sqn{\z^k - \z^*}} 
	\\&\quad
	- \gamma\E{\sqn{\w^k - \z^{k+1}}}  - (1-\gamma)\E{\sqn{\z^{k+1} - \z^k}} 
	\\&\quad
	- 2\eta\E{\<\Ek{\Delta^k}  - F(\z^*), \z^{k+1} - \z^*>}
	\\&\quad
	- 2\eta\E{\<\Delta^k - \Ek{\Delta^k}, \z^{k+1} - \z^k>}
	\\&\quad
	- 2\eta\E{\<\Delta^k - \Ek{\Delta^k}, \z^{k} - \z^*>}.
	\end{align*}
Using the tower property of expectation we can obtain the following:
	\begin{align*}
		\E{\<\Ek{\Delta^k} - \Delta^k, \z^k - \z^*>} &= \E{\Ek{\<\Ek{\Delta^k} - \Delta^k, \z^k - \z^*>}}
		\\&=
		\E{{\<\Ek{\Ek{\Delta^k} - \Delta^k}, \z^k - \z^*>}}
		\\&=
		\E{{\<\Ek{\Delta^k} - \Ek{\Delta^k}, \z^k - \z^*>}} = 0.
	\end{align*}
Hence, with \eqref{vrvi:eq:3}
	\begin{align*}
	\E{\sqn{\z^{k+1} - \z^*}}
	&\leq
	\E{\sqn{\z^k - \z^*}} + \gamma\E{\sqn{\z^k - \z*}} - \gamma\E{\sqn{\z^k - \z^*}}  
	\\&\quad
	- \gamma\E{\sqn{\w^k - \z^{k+1}}} - (1-\gamma)\E{\sqn{\z^{k+1} - \z^k}} 
	\\&\quad - 2\eta\E{\<\F(\z^k) + \alpha(\F(\z^k) - \F(\z^{k-1})) - \F(\z^*), \z^{k+1} - \z^*>}
	\\&\quad+ 2\eta\E{\<\Ek{\Delta^k} - \Delta^k , \z^{k+1} - \z^k>}.
	\end{align*}
Using the Young's inequality we get
	\begin{align*}
	\E{\sqn{\z^{k+1} - \z^*}}
	&\leq
	\E{\sqn{\z^k - \z^*}} + \gamma\E{\sqn{\w^k - \z^*}} - \gamma\E{\sqn{\z^k - \z^*}} 
	\\& \quad
	- \gamma\E{\sqn{\w^k - \z^{k+1}}} - (1-\gamma)\E{\sqn{\z^{k+1} - \z^k}} 
	\\& \quad - 2\eta\E{\<\F(\z^k) + \alpha(\F(\z^k) - \F(\z^{k-1})) - \F(\z^*), \z^{k+1} - \z^*>}
	\\& \quad+ 2\eta^2\E{\sqn{\Ek{\Delta^k} - \Delta^k}}  + \frac{1}{2}\E{\sqn{\z^{k+1} - \z^k}}
    \\&= \E{\sqn{\z^k - \z^*}} + \gamma\E{\sqn{\w^k - \z^*}} - \gamma\E{\sqn{\z^k - \z^*}}  
	\\& \quad
	- \gamma\E{\sqn{\w^k - \z^{k+1}}} - (1/2-\gamma)\E{\sqn{\z^{k+1} - \z^k}} 
	\\&\quad
	+ 2\eta^2\E{\sqn{\Ek{\Delta^k} - \Delta^k}}
	\\&\quad - 2\eta\E{\<\F(\z^k) + \alpha(\F(\z^k) - \F(\z^{k-1})) - \F(\z^*), \z^{k+1} - \z^*>}.
	\end{align*}
By \eqref{vrvi:eq:1} we get
	\begin{align*}
	\E{\sqn{\z^{k+1} - \z^*}}
	&\leq
	\E{\sqn{\z^k - \z^*}} + \gamma\E{\sqn{\w^k - \z^*}} - \gamma\E{\sqn{\z^k - \z^*}}  
	\\&\quad
	- \gamma\E{\sqn{\w^k - \z^{k+1}}}
	- (1/2-\gamma)\E{\sqn{\z^{k+1} - \z^k}}  
	\\&\quad + 4\delta^2 \eta^2\E{\sqn{\z^{k} - \w^{k-1}}} + 4\alpha^2 \delta^2\eta^2 \E{\sqn{\z^k- \z^{k-1}}}
	\\&\quad - 2\eta\E{\<\F(\z^k) + \alpha(\F(\z^k) - \F(\z^{k-1})) - \F(\z^*), \z^{k+1} - \z^*>}
	\\&=\E{\sqn{\z^k - \z^*}} - 2\eta\E{\<\F(\z^{k+1}) - \F(\z^*), \z^{k+1} - \z^*>} 
	\\& \quad
	+ \gamma\E{\sqn{\w^k - \z^*}} - \gamma\E{\sqn{\z^k - \z^*}} - \gamma\E{\sqn{\w^k - \z^{k+1}}}
	\\& \quad
	- (1/2-\gamma)\E{\sqn{\z^{k+1} - \z^k}}  
	\\& \quad  + 4\delta^2 \eta^2\E{\sqn{\z^{k} - \w^{k-1}}} + 4\alpha^2 \delta^2 \eta^2 \E{\sqn{\z^k- \z^{k-1}}} 
	\\& \quad - 2\eta\E{\<\F(\z^k) - \F(\z^{k+1}) + \alpha(\F(\z^k) - \F(\z^{k-1})), \z^{k+1} - \z^*>}.
	\end{align*}
Assumption \ref{as:strmon} about $\mu$-strong monotonicity of $F$ gives
	\begin{align*}
	&\E{\sqn{\z^{k+1} - \z^*}} 
	\\&\leq
	\E{\sqn{\z^k - \z^*}} - 2\mu\eta\E{\sqn{\z^{k+1} - \z^*}} + \gamma\E{\sqn{\w^k - \z^*}} - \gamma\E{\sqn{\z^k - \z^*}} 
	\\&\quad
	- \gamma\E{\sqn{\w^k - \z^{k+1}}}  - (1/2-\gamma)\E{\sqn{\z^{k+1} - \z^k}}
	\\&\quad
	+ 4\delta^2 \eta^2\E{\sqn{\z^{k} - \w^{k-1}}} + 4\alpha^2 \delta^2 \eta^2 \E{\sqn{\z^k- \z^{k-1}}}
	\\&\quad - 2\eta \E{\<\F(\z^k) - \F(\z^{k+1}) + \alpha(\F(\z^k) - \F(\z^{k-1})), \z^{k+1} - \z^*>}
	\\&=\E{\sqn{\z^k - \z^*}} - 2\mu\eta\E{\sqn{\z^{k+1} - \z^*}} + \gamma\E{\sqn{\w^k - \z^*}} - \gamma\E{\sqn{\z^k - \z^*}} 
	\\& \quad
	- \gamma\E{\sqn{\w^k - \z^{k+1}}}  - (1/2-\gamma)\E{\sqn{\z^{k+1} - \z^k}} 
	\\&\quad
	+ 4\delta^2 \eta^2 \E{\sqn{\z^{k} - \w^{k-1}}} + 4\alpha^2 \delta^2 \eta^2 \E{\sqn{\z^k- \z^{k-1}}}
	\\&\quad -2\eta\E{\<\F(\z^k) - \F(\z^{k+1}), \z^{k+1} - \z^*>} 
	\\&\quad -2\alpha\eta\E{\<\F(\z^k) - \F(\z^{k-1}), \z^{k} - \z^*>}
	\\&\quad -2\alpha\eta\E{\<\F(\z^k) - \F(\z^{k-1}), \z^{k+1} - \z^k>}.
	\end{align*}
With small rearrangement and Young's inequality we get
	\begin{align*}
	&\E{\sqn{\z^{k+1} - \z^*}} + 2\eta\E{\<\F(\z^k) - \F(\z^{k+1}), \z^{k+1} - \z^*>}
	\\&\leq
	\E{\sqn{\z^k - \z^*}} - 2\mu\eta\E{\sqn{\z^{k+1} - \z^*}} + \gamma\E{\sqn{\w^k - \z^*}} 
	\\&\quad 
	- \gamma\E{\sqn{\z^k - \z^*}} +\alpha \cdot 2\eta\E{\<\F(\z^{k-1}) - \F(\z^{k}), \z^{k} - \z^*>}
	\\& \quad
	- \gamma\E{\sqn{\w^k - \z^{k+1}}}  - (1/2-\gamma)\E{\sqn{\z^{k+1} - \z^k}} 
	\\&\quad
	+ 4\delta^2 \eta^2 \E{\sqn{\z^{k} - \w^{k-1}}} + 4\alpha^2 \delta^2 \eta^2 \E{\sqn{\z^k- \z^{k-1}}}
	\\&\quad 
	+ 4\alpha^2 \eta^2 \E{\sqn{\F(\z^k) - \F(\z^{k-1})}} +\frac{1}{4}\E{\sqn{\z^{k+1} - \z^k}}.
	\end{align*}
With $L$-Lipshitzness of $\F$ (Assumption \ref{as:Lipsh}) and $\gamma \leq \frac{1}{8}$, we have
	\begin{align*}
	&\E{\sqn{\z^{k+1} - \z^*}} + 2\eta\E{\<\F(\z^k) - \F(\z^{k+1}), \z^{k+1} - \z^*>}
	\\&\leq
	\E{\sqn{\z^k - \z^*}} - 2\mu\eta\E{\sqn{\z^{k+1} - \z^*}} + \gamma\E{\sqn{\w^k - \z^*}} 
	\\&\quad 
	- \gamma\E{\sqn{\z^k - \z^*}} +\alpha \cdot 2\eta\E{\<\F(\z^{k-1}) - \F(\z^{k}), \z^{k} - \z^*>}
	\\& \quad
	- \gamma\E{\sqn{\w^k - \z^{k+1}}} + 4\delta^2 \eta^2 \E{\sqn{\z^{k} - \w^{k-1}}} - \frac{1}{8}\E{\sqn{\z^{k+1}- \z^{k}}}
	\\&\quad
	 + 4\alpha^2 \delta^2 \eta^2 \E{\sqn{\z^k- \z^{k-1}}}
	+ 4\alpha^2 L^2 \eta^2  \E{\sqn{\z^k- \z^{k-1}}}.
	\end{align*}
Now, we add $\tfrac{\gamma +\eta\mu}{\gamma}\E{\sqn{\w^{k+1} - \z^*}}$ to both sides and use update for $w^{k+1}$ (lines \ref{alg1:wk_zk} and \ref{alg1:wk_wk} of Algorithm \ref{alg:optmasha})
	\begin{align*}
	\frac{\gamma + \eta\mu}{\gamma}\E{\mathbb{E}_{w^{k+1}}{\sqn{\w^{k+1} - \z^*}}} &= (\gamma + \eta\mu)\E{\sqn{\z^{k} - \z^*}}
	\\&\quad
	+ \frac{(\gamma + \eta\mu)(1-\gamma)}{\gamma}\E{\sqn{\w^{k} - \z^*}},
	\end{align*}
and get 
	\begin{align*}
	(1+ 2 \mu \eta)&\E{\sqn{\z^{k+1} - \z^*}} + \frac{\gamma + \eta\mu}{\gamma}\E{\sqn{\w^{k+1} - \z^*}} 
	\\&\quad
	+ 2\eta\E{\<\F(\z^k) - \F(\z^{k+1}), \z^{k+1} - \z^*>}
	\\&\quad
	+ \gamma\E{\sqn{\w^k - \z^{k+1}}} +\frac{1}{8} \E{\sqn{\z^{k+1}- \z^{k}}}
	\\&\leq
	(1+\eta\mu)\E{\sqn{\z^k - \z^*}} + \gamma\E{\sqn{\w^k - \z^*}} 
	\\&\quad
	+ \frac{(\gamma + \eta\mu)(1-\gamma)}{\gamma}\E{\sqn{\w^{k} - \z^*}}
	\\&\quad	
	+\alpha \cdot 2\eta\E{\<\F(\z^{k-1}) - \F(\z^{k}), \z^{k} - \z^*>}
	\\& \quad
	 + \frac{4\delta^2 \eta^2}{\gamma} \cdot \gamma \E{\sqn{\w^{k-1} - \z^{k}}}
	\\&\quad
	+ 32\alpha^2 (\delta^2 + L^2) \eta^2 \cdot \frac{1}{8} \E{\sqn{\z^k- \z^{k-1}}}.
	\end{align*}
Note that $\eta \leq \min\left\{\frac{\sqrt{\alpha\gamma}}{2 \delta}, \frac{1}{8(L + \delta)}\right\}$, $0 <\alpha< 1$, $ 0 <\gamma < 1$ and $L \geq \mu$, then we get that
$$
\frac{4\delta^2 \eta^2}{\gamma} \leq \alpha; \quad 32\alpha^2 (\delta^2 + L^2) \eta^2  \leq \frac{1}{2}; \quad (1 + \mu \eta) \leq \left(1- \frac{\mu \eta}{2}\right) (1 + 2 \mu \eta);
$$
$$
\gamma + \frac{(\gamma + \eta\mu)(1-\gamma)}{\gamma} = \left(1- \frac{1}{\frac{1}{\eta\mu} + \frac{1}{\gamma}}\right) \frac{\gamma + \eta\mu}{\gamma}
$$
Hence, it holds
	\begin{align*}
	(1+ 2 \mu \eta)&\E{\sqn{\z^{k+1} - \z^*}} + \frac{\gamma + \eta\mu}{\gamma}\E{\sqn{\w^{k+1} - \z^*}} 
	\\&\quad
	+ 2\eta\E{\<\F(\z^k) - \F(\z^{k+1}), \z^{k+1} - \z^*>}
	\\&\quad
	+ \gamma\E{\sqn{\w^k - \z^{k+1}}} +\frac{1}{8} \E{\sqn{\z^{k+1}- \z^{k}}}
	\\&\leq
	\left(1- \frac{\mu \eta}{2}\right) \cdot (1+2\eta\mu)\E{\sqn{\z^k - \z^*}} 
	\\&\quad
	+ \left(1- \frac{1}{\frac{1}{\eta\mu} + \frac{1}{\gamma}}\right) \cdot \frac{\gamma + \eta\mu}{\gamma}\E{\sqn{\w^{k} - \z^*}}
	\\&\quad	
	+\alpha \cdot 2\eta\E{\<\F(\z^{k-1}) - \F(\z^{k}), \z^{k} - \z^*>}
	\\& \quad
	 + \alpha \cdot \gamma \E{\sqn{\w^{k-1} - \z^{k}}}
	\\&\quad
	+ \frac{1}{2} \cdot \frac{1}{8} \E{\sqn{\z^k- \z^{k-1}}}.
	\end{align*}
Definition \eqref{eq:Lf_fixed} of the Lyapunov function move us to
    \begin{align*}
	\E{\Psi^{k+1}} \leq
	\max\left[ \left(1- \frac{\mu \eta}{2}\right); \left(1- \frac{1}{\frac{1}{\eta\mu} + \frac{1}{\gamma}}\right) ; \alpha ; \frac{1}{2} \right] \cdot \E{\Psi^k}.
	\end{align*}
	It remains to show, that $$\Psi^k \geq \frac{1}{2}\sqn{\z^{k} - \z^*}.$$
	\begin{align*}
		\Psi^k&\geq \sqn{\z^{k} - \z^*}
		+
		\frac{1}{8}\sqn{\z^{k} - \z^{k-1}}
		+
		2\eta \<\F(\z^{k}) -\F(\z^{k-1}) ,\z^*-\z^k>
		\\&\geq 
		\sqn{\z^{k} - \z^*}
		+
		\frac{1}{8}\sqn{\z^{k} - \z^{k-1}}
		-
		\frac{1}{2}\sqn{\z^k - \z^*}
		-
		2\eta^2\sqn{\F(\z^k) - \F(\z^{k-1})}.
	\end{align*}
	Using $L$-Lipschitzness of $\F$ we get 
	\begin{align*}
	\Psi^k&\geq \frac{1}{2}\sqn{\z^k - \z^*} + \frac{1}{8} \left( 1 - 16 L^2 \eta^2 \right)\sqn{\z^{k} - \z^{k-1}}.
	\end{align*}
	With $\eta \leq \frac{1}{8 L}$, we get $$\Psi^k \geq \frac{1}{2}\sqn{\z^{k} - \z^*}.$$
\end{proof}

\begin{theorem}[Theorem \ref{th:ALg1_conv}]\label{th:app_fixed}
	Consider the problem \eqref{eq:VI} under Assumptions~\ref{as:Lipsh}, \ref{as:strmon} and \ref{as:delta}. Let  $\{\z^k\}$ be the sequence generated by Algorithm~\ref{alg:optmasha} with compressors from Definition \ref{def:PermK} and parameters 
	\begin{align*}
		0 < \gamma  \leq \frac{1}{8}, \quad \alpha = \frac{1}{2}, \quad \eta = \min\left\{\frac{\sqrt{\alpha\gamma}}{2 \delta}, \frac{1}{8(L + \delta)}\right\}.
	\end{align*}
	Then, given $\varepsilon>0$, the number of iterations for 
$\|\z^k - \z^*\|^2 \leq \varepsilon$ is 
\begin{equation*}
    O\left( \left[\frac{1}{\gamma} +\frac{L}{\mu} + \frac{\delta}{ \sqrt{\gamma}\mu}\right] \log \frac{1}{\varepsilon} \right).
\end{equation*}
\end{theorem}
\begin{proof}
From Lemma \ref{lem:fixed_key} we can get that the iteration complexity of Algorithm~\ref{alg:optmasha}:
\begin{align*}
O\left( \left[1 + \frac{1}{\eta \mu}+\frac{1}{\gamma}\right] \log \frac{1}{\varepsilon} \right) &= O\left( \left[\frac{1}{\gamma} + \frac{\delta}{\mu} + \frac{L}{\mu} + \frac{\delta}{ \sqrt{\gamma}\mu}\right] \log \frac{1}{\varepsilon} \right) \\
&= O\left( \left[\frac{1}{\gamma} +\frac{L}{\mu} + \frac{\delta}{ \sqrt{\gamma}\mu}\right] \log \frac{1}{\varepsilon} \right).
\end{align*}
\end{proof}

\end{document}